\documentclass{amsart}
\usepackage{amsmath, amssymb, amsgen, amsthm, amscd, xspace}
\newcommand{\Comp}{\mathbb{C}}

\newcommand{\Int}{\mathbb{Z}}
\newcommand{\mk}{\mathbb{K}}

\newcommand{\eps}{\epsilon}
\newcommand{\lam}{\lambda}
\newcommand{\del}{\delta}

\newcommand{\sfx}{\mathsf{X}}
\newcommand{\sfd}{\mathsf{D}}
\newcommand{\sfe}{\mathsf{E}}
\newcommand{\sfep}{\sfe'}
\newcommand{\sff}{\mathsf{F}}
\newcommand{\sft}{\mathsf{T}}
\newcommand{\sftp}{\mathsf{T}'}
\newcommand{\sfh}{\mathsf{H}}
\newcommand{\sfm}{\mathsf{M}}
\newcommand{\Mnk}{\sfm_{n,k}}
\newcommand{\Mkn}{\sfm_{k,n}}
\newcommand{\dd}{\partial}

\newcommand{\mat}{\mathrm{M}}

\newcommand{\sump}{\sum{\!\!}^{\prime}}

\newcommand{\gl}{{\mathfrak{gl}}}
\newcommand{\gln}{{\gl_n}}
\newcommand{\ssl}{{\mathfrak{sl}}}
\newcommand{\sln}{{\ssl_n}}
\newcommand{\glN}{{\gl_N}}
\newcommand{\glk}{{\gl_k}}

\newcommand{\go}{{\mathfrak{o}}}
\newcommand{\oN}{{\go_N}}
\newcommand{\spg}{{\mathfrak{sp}}}

\newcommand{\spN}{{\spg_{N}}}
\renewcommand{\gg}{{\mathfrak{g}}}

\newcommand{\gp}{{\gg'}}

\newcommand{\zg}{\mathcal{Z}(\gg)}

\newcommand{\ugn}{U(\gln)}

\newcommand{\ug}{U(\gg)}

\newcommand{\ga}{\mathfrak{a}}
\newcommand{\car}{\mathfrak{h}}
\newcommand{\card}{{\car}^*}
\newcommand{\bor}{\mathfrak{b}}
\newcommand{\pp}{\mathfrak{p}}
\newcommand{\nil}{\mathfrak{n}}
\newcommand{\lev}{\mathfrak{m}}
\newcommand{\zz}{\mathfrak{z}}

\newcommand{\A}{\mathcal{A}}
\newcommand{\weyl}{\mathcal{W}}

\newcommand{\tr}{\operatorname{tr}}
\newcommand{\sgn}{\operatorname{sgn}}

\newcommand{\Ad}{\operatorname{Ad}}
\newcommand{\ann}{\operatorname{Ann}}

\newcommand{\End}{\operatorname{End}}

\newtheorem{theorem}[equation]{Theorem}
\newtheorem{cor}[equation]{Corollary}
\newtheorem{prop}[equation]{Proposition}

\theoremstyle{remark}
\newtheorem{remark}[equation]{Remark}

\theoremstyle{definition}
\newtheorem{definition}[equation]{Definition}
\newtheorem{example}[equation]{Example}

\numberwithin{equation}{section}
\allowdisplaybreaks[1]
\newlength{\originalbase}
\begin{document}
%
%
%
\title{Minimal polynomials of simple highest weight modules over classical
Lie algebras}
\author{Victor Protsak}
\address{Department of Mathematics, Cornell University, Ithaca, NY 14853}
\email{\tt protsak@math.cornell.edu}
\date{January 23, 2008}
%
%
\begin{abstract}
We completely determine the minimal polynomial of an arbitrary 
simple highest weight module $L(\lambda)$ over a complex 
classical Lie algebra $\mathfrak{g}\subseteq\mathfrak{gl}_N$ 
relative to its defining module $\pi=\mathbb{C}^{N}$. These
results are applied to ordering on primitive ideals and algebraic
properties of Howe duality correspondence.
\end{abstract}
\subjclass[2000]{Primary 17B35; Secondary 22E46}
\maketitle
%
%
\section{Introduction}
%
%
Let $\gg$ be a complex reductive Lie algebra and $\pi$ be
a finite-dimensional $\gg$-module. 
The \emph{minimal polynomial} of a $\gg$-module $V$
is a polynomial $q_{\pi,V}$ in one variable which records certain matrix
identities involving generators of the Lie algebra $\gg$ acting 
on $V$. Minimal polynomials were explicitly introduced by Oshima
in \cite{Oshima_min}, but their theory goes back to 
the work of Bracken and Green, 
O'Brien, Cant and Carey, and Gould on \emph{characteristic identities},
see \cite{Gould_characteristic} and the references therein.
We consider minimal polynomials in the setting of a classical
Lie algebra $\gg\subseteq\glN$ with the defining module $\pi=\Comp^N$
and completely determine the minimal polynomial $q_{\pi,V}$ for an arbitrary 
{simple highest weight module} $V=L(\lam)$.
Here $\gg=\glN, \oN$, or $\spN$ (in the symplectic case, $N$ is even). 
This generalizes many earlier results: an explicit formula for $q_{\pi,V}$ 
when $V=M(\lam)$ ia a Verma module, \cite{Gould_characteristic}, 
and Oshima and Oda's description of
the minimal polynomials of $V=M_{\pp}(\lam)$, the 
{generalized Verma module of scalar type} (under some technical restrictions), 
\cite{Oda_Oshima}. 

To state our results in full would require more notation, but 
they are easy to describe in a qualitative form. 
The roots of the minimal polynomial of $L(\lam)$ 
are expressed in terms of the \emph{shuffle decomposition}, 
a combinatorial object associated to the highest weight $\lam$. 
In the general linear case, it is described via a simplified version 
of the Robinson--Shensted--Knuth algorithm. The 
symplectic and orthogonal cases are analogous, but somewhat more involved.
The defining module plays a very distinguished role 
in invariant theory and representation theory of a 
classical Lie group $G$: it is
in many ways the smallest and the simplest module, and it is very useful
in constructing and studying finite-dimensional $G$-modules.
Our results may be viewed as a manifestation of this principle 
in the context of minimal polynomials and (infinite-dimensional) 
highest weight modules over the Lie algebra $\gg$.

There are several important reasons to study minimal polynomials: 
they are related to explicit description of primitive ideals in $\ug$,
noncommutative Capelli identities, the effect of Howe duality on ideals, 
\cite{Protsak_Transfer} and Section \ref{sec:transfer},  
the ordering on the primitive ideals and 
the multiplicity of the adjoint representation in the primitive quotients of
$\ug$, Section \ref{sec:primitive}. It is easy to see that for a fixed $\pi$, 
the minimal polynomial of a module $V$ depends only on its annihilator 
ideal $I=\ann_{\ug}V$ in the universal enveloping algebra
$\ug$. Hence it makes sense to speak about the minimal polynomial $q_I$
of an ideal. Moreover, if $I\subseteq J$ is an inclusion of ideals 
then the polynomial $q_J$ divides $q_I$.
By Duflo's theorem, \cite{Duflo, JAN}, 
any primitive ideal of $\ug$, where
$\gg$ is a complex reductive Lie algebra, is the annihilator of 
a simple highest weight module $L(\lam)$. Hence our main theorems
yield information about the ordering on the primitive spectrum of the 
algebra $\ug$ for a classical Lie algebra $\gg$, some of it new. 
It turns out that for $L(\lam)$ having the infinitesimal character
of the trivial representation, the knowledge of the minimal polynomial 
is equivalent to the knowledge of the \emph{$\tau$-invariant} 
of the module $L(\lam)$, but for a singular highest weight $\lam$, 
the minimal polynomial can be a stronger invariant. 

Another motivation for studying the minimal polynomial comes
from the theory of the reductive dual correspondence 
(or \emph{Howe duality}),  \cite{Howe_Theta, Howe_Remarks}. 
The relation between the centers of
the universal enveloping algebras of two Lie algebras forming
a reductive dual pair is fairly well understood 
through the theory of \emph{Capelli identities}, 
\cite{Howe_Umeda, Itoh_Advances, Molev_Nazarov}. It has important 
applications to harmonic analysis on reductive Lie groups,  
\cite{Przebinda_CapelliHC}.
In \cite{Protsak_Transfer} we have studied a more general situation: the effect
of Howe duality on a general, not necessarily centrally generated, ideal.
Here we prove a simple relation between the minimal polynomials
of the modules $V$ and $V'$ that correspond to each other 
under an algebraic version of the Howe duality. 

This paper is organized as follows. In Section \ref{sec:projection} 
we collect
some mostly known facts about the Harish-Chandra projection
map in the relative setting. The main result there, 
Proposition \ref{prop:criterion}, is 
a criterion to determine whether a given $ad(\gg)$-invariant subspace 
$T\subseteq \ug$ annihilates $L(\lam)$. 
In Section \ref{sec:minimal} 
we define the minimal polynomial, Definition \ref{def:minpol},
and explain our approach to computing
it through the use of formal resolvents, Proposition \ref{prop:technique}.
The next two sections contain our main results, first in the general 
linear case, Theorem \ref{thm:minpol_gl},
and then in the symplectic and orthogonal cases, 
Theorem \ref{thm:minpol_gg}. A key role in our analysis is played by
a certain relative Harish-Chandra projection map that allows us to
relate minimal polynomials for modules over $\gln$ and 
$\gl_{n-1}$ and apply induction. For an irreducible
reductive dual pair of Lie algebras $(\gg, \gp)$, we establish
in Section \ref{sec:transfer}
a relation between the minimal polynomials of modules over $\gg$
and $\gp$ in an algebraic Howe duality with each other.
Section \ref{sec:primitive} is devoted to applications to primitive ideals.
%
%
\section{Harish-Chandra projection}\label{sec:projection}
%
In this paper, the ground field is $\Comp$, although many results
extend to the case of an arbitrary field $\mk$ of characteristic zero. 
The case of the field $\Comp(t)$ of rational functions in several variables
seems especially promising for applications to the study of prime
ideals in enveloping algebras. Such extensions hold
under assumption that Lie algebras are split over $\mk$, with little
change in the proofs.

Let $\gg$ be a split reductive Lie algebra, $\car$
a Cartan subalgebra, $R\subset\card$ be the root system of $(\gg,\car)$
and $\gg_{\alpha}$ be the root subspace corresponding to $\alpha\in R$. 
Choose a system of simple roots $B\subset R$ and denote by $R^+$ and $R^-$
the corresponding subsets of $R$ consisting of positive and negative roots. 
For any subset $S\subset B$ we let $R_S=\Int S\cap R$. Then $R_S$ is
a root system and we consider the following subspaces of $\gg$:
\begin{displaymath}
\nil^+_S=\bigoplus_{\alpha\in R^+\setminus R_S}\gg_{\alpha}, \quad
\nil^-_S=\bigoplus_{\alpha\in R^-\setminus R_S}\gg_{\alpha}, \quad
\lev_S=\car\oplus\bigoplus_{\alpha\in R_S}\gg_{\alpha}.
\end{displaymath}
It is well known that $\pp_S=\lev_S\oplus\nil^+_S$ is a standard
parabolic of $\gg$ and all standard parabolics arise in this way.
In particular, for $S=\emptyset$ we get $\lev_S=\car, \pp_S=\bor$
is the Borel subalgebra, and $\nil^+_S=\nil$ is its nilradical.

Denote by $U(\gg)$ the universal enveloping algebra of $\gg$.
The Lie algebra $\gg$ has a triangular decomposition
$\gg=\nil^{+}_{S}\oplus\lev_{S}\oplus\nil^{-}_{S}$ and the 
Poincare--Birkhoff--Witt (PBW) theorem gives rise to a canonical isomorphism 
\begin{equation}
\ug\simeq (\nil^-_S\ug+\ug\nil^+_S)\oplus U(\lev_S). \label{eq:decomp}
\end{equation}
Let $\zz_S$ be the center of the Levi factor $\lev_S$. Then
$\zz_S$ is the subspace of $\car$ which is the annihilator
of $S\subset\car^*$. It is clear that the adjoint action of
$\zz_S$ on $\ug$ is semisimple, and the subspace of invariants
under this action is a subalgebra $\ug^{\zz_S}\subset\ug$
spanned by the $ad(\car)$-weight vectors of $\ug$
whose weight belongs to $\Int S$. 
\begin{definition}
The relative Harish-Chandra projection map 
$pr_{\gg/\lev}: \ug\to U(\lev)$ is the projection onto the second
summand in the decomposition \eqref{eq:decomp}.
\end{definition}
The map $pr_{\gg/\car}$, which we call the (absolute) 
Harish-Chandra projection map, is well known in representation theory.
Its restriction to the center $\zg$ of the
universal enveloping algebra $\ug$ is an injective algebra homomorphism
from $\zg$ to $U(\car)$ whose image consists of $W$-invariants
in $U(\car)$ with respect to the shifted action of the Weyl
group $W$. Composing $pr_{\gg/\car}$ with the algebra automorphism
of $U(\car)\simeq P(\card)$ induced by the $\rho$-shift, one
obtains the \emph{Harish-Chandra homomorphism}, in fact, an isomorphism
between $\zg$ and the algebra $P(\card)^W$ of $W$-invariant 
polynomial functions on $\card$. The Harish-Chandra homomorphism arises
from the action of $\zg$ on highest weight modules as follows.
Let $M$ be a highest weight module over $\gg$ with highest weight
vector $v_\lam$. Thus $M$ is generated by $v_{\lam}$ as a $\ug$-module,
$\nil v_{\lam}=0$, and $hv_{\lam}=\lam(h)v_{\lam}$ for $h\in\car$. Denote
by $M_{<}$ the linear span of all weight vectors in $M$ of weight 
lower than $\lam$, then $M=\Comp v_{\lam}\oplus M_{<}$. 
Let $v^*$ be the linear 
functional on $M$ which vanishes on $M_{<}$ and such that 
$\langle v^*,v\rangle=1$. Then for any element $u\in\ug$, we have
an equality
\begin{equation} \label{eq:projhw}
\langle v^*, u v_{\lam}\rangle = 
\langle v^*, pr_{\gg/\car}(u)v_{\lam}\rangle.
\end{equation}
The associative algebra $U(\car)$ is canonically identified with the
commutative algebra of polynomial functions on $\card$, 
and we denote by $ev_\lam$ the homomorphism from $U(\car)$ to $\Comp$ which
is the evalution at $\lam\in\card$.
The identity \eqref{eq:projhw} says that for any $u\in\ug$, 
the coefficient of $v_{\lam}$ in the expansion of $u v_{\lam}$ with 
respect to a basis of $M$ consisting of weight vectors is
$ev_\lam(pr_{\gg/\car}(u))$. 

A few elementary properties of the relative projection map are collected
in the following proposition (cf. \cite{JAN}, 5.12).
\begin{prop} \label{prop:projprop}
Fix a subset $S\subset B$ and let $\lev=\lev_S$. 
\begin{enumerate}
\item
The map $pr_{\gg/\lev}$ is a homomorphism of $(U(\lev),U(\lev))$-bimodules.
Its restriction to the subalgebra  $\ug^{\zz}\subset\ug$ 
is an algebra homomorphism.
\item
If $u\in\nil^-_S\ug$ or $u\in\ug\nil^+_S$ then $pr_{\gg/\lev}(u)=0$.
\item
If $u\in\ug$ is an $ad(\car)$-eigenvector of weight $\mu$ and 
$\mu\notin\Int S$ then $pr_{\gg/\lev}(u)=0$. 
\end{enumerate}
\end{prop}
\begin{proof}
Since the adjoint action of $\lev_S$ on $\ug$ stabilizes 
$\nil^-_S,\nil^+_S$, and $\lev_S$, we see that both summands 
in the decomposition \eqref{eq:decomp} are invariant under
the left and right multiplication by $\lev_S$, hence
they are $(U(\lev_S),U(\lev_S))$ subbimodules of $\ug$. 
Assertion (1) follows.

Assertion (2) is immediate from the definition. Since
both summands in the decomposition \eqref{eq:decomp} are
also $ad(\car)$-invariant, $pr_{\gg/\lev}$ is an $ad(\car)$-equivariant
map. The weights of $ad(\car)$ on $U(\lev_S)$ are precisely
$\Int S$, which implies assertion (3). (When $S=\emptyset$, we
let $\Int S=\{0\}$.)
\end{proof}
For any $\lam\in\card$ there is a unique up to an isomorphism
simple highest weight module $L(\lam)$ with highest weight $\lam$.
The next proposition gives a powerful criterion for determining whether 
an $ad(\gg)$-invariant subspace of $\ug$ annihilates $L(\lam)$.
\begin{prop} \label{prop:criterion}
Let $P\subset\ug$ be an $ad(\gg)$-invariant subspace, $P_0$
the zero weight space of $P$, and $\lam\in\card$ a highest weight.
Then the following conditions are equivalent:
\begin{enumerate}
\item
$P$ annihilates $L(\lam)$;
\item
$pr_{\gg/\car}(p)$ vanishes at $\lam$ for all $p\in P$;
\item
$pr_{\gg/\car}(p)$ vanishes at $\lam$ for all $p\in P_0$.
\end{enumerate}
\end{prop}
\begin{proof}
Clearly, (2) implies (3). We will prove that (3) implies (1) and
(1) implies (2). Suppose that $P$ annihilates $L(\lam)$, then 
it follows in particular that $Pv_\lam=0$. From \eqref{eq:projhw}
$pr_{\gg/\car}(P) v_\lam=0$, therefore (2) holds.

The $ad(\gg)$-invariance of $P$ implies that for any subspace 
$\ga\subseteq\gg, \ga P=P\ga+P$. By the Poincare--Birkhoff--Witt theorem,
$\ug=U(\nil^{-})U(\bor),$ so the two-sided ideal $\langle P\rangle=\ug P\ug$ 
generated by the subspace $P$ is equal to $U(\nil^{-})PU(\bor)$. Applying 
the elements of this ideal to the highest weight vector $v_\lam$, we find that
\begin{equation*}
\langle P\rangle L(\lam)=\ug P\ug v_\lam=U(\nil^{-})PU(\bor)v_\lam=
U(\nil^{-})Pv_\lam.
\end{equation*}
Now let us decompose $P$ as $P_{+}\oplus P_0\oplus P_{-}$, where $P_{+}$ 
(respectively, $P_{-}$) is spanned by the positive (respectively, negative)
$ad(\car)$-weight vectors in $P$. The action of $P_{+}$ annihilates $v_\lam$, 
so $Pv_\lam=pr_{\gg/\car}(P_0)v_\lam+L(\lam)_{<}$. Suppose that 
condition (3) holds, so that $pr_{\gg/\car}(P_0)$ acts by $0$ on
$v_\lam$, then the submodule 
$\langle P\rangle L(\lam)$ of $L(\lam)$ is contained in $L(\lam)_{<}$. 
Therefore, it is a proper submodule of a simple module, 
implying that $\langle P\rangle$ annihilates $L(\lam)$. 
\end{proof}
The second condition  of the proposition says that $\lam$ is
in Joseph's ``characteristic variety'' of $P$. 
In a different language,
the weight $\lam$ appears in  the Jacquet functor of the 
Harish-Chandra bimodule $\ug/\langle P\rangle$ 
with respect to the nilradical $(\nil^{-},\nil^{+})\subset(\gg,\gg)$,
cf \cite{Joseph_Characteristic, Wallach}. 

Another important property of Harish-Chandra projections
is the \emph{composition formula}. 
\begin{prop}\label{prop:composition}
Let $S\subseteq B$ and
suppose that the root system $R_S$ admits a basis (i.e. a system
of simple roots) $B_S\subseteq B$. If $S'\subseteq B_S$ then
we may construct a Levi subalgebra $\lev'=\lev_{S'}$ of the reductive Lie 
algebra $\lev$ as above and define the Harish-Chandra projection 
$pr_{\lev/\lev'}: U(\lev)\to U(\lev')$. Moreover, $\lev'$ is also
the Levi subalgebra of the Lie algebra $\gg$ itself associated with the
subset $S'\subseteq B$, so that the map 
$pr_{\gg/\lev'}: \ug\to U(\lev')$ is defined. Then the following
composition formula holds:
\begin{equation}\label{eq:composition}
pr_{\lev/\lev'}\circ pr_{\gg/\lev}=pr_{\gg/\lev'}.
\end{equation}
\end{prop}
\begin{proof}
This follows immediately from the definitions.
\end{proof}
%
%
\section{Minimal polynomials and formal resolvents}\label{sec:minimal}
%
%
We begin with a very general notion of {minimal polynomial} for the
elements of an associative, but possibly noncommutative, algebra.
\begin{definition}\label{def:genminpol}
Let $R$ be an associative algebra over $\Comp$.
For any element $a\in R$, the non-zero monic polynomial 
$q\in\Comp[u]$ of minimal degree such that $q(a)=0$ is
called \emph {the minimal polynomial} of the element $a$.  
\end{definition}
It is clear that all polynomials $f\in\Comp[u]$ annihilating a fixed 
element $a\in R$ form a (possibly zero) ideal. Since the ring of 
polynomials in one variable over a field is a principal ideal domain, 
we conclude that $f(a)=0$ holds for a polynomial $f\in\Comp[u]$ 
if and only if $f$ is divisible by the minimal polynomial of $a$. 
In particular, if a minimal polynomial exists, it is unique.

This notion of minimal polynomial, which can be defined for 
associative algebras over any field, simultaneously generalizes minimal 
polynomials in field theory and minimal polynomials of matrices. 
\begin{definition}\label{def:resolvent}
Let $a$ be an element of an associative algebra $R$.
We call the formal power series 
\begin{equation}\label{eq:resolve}
\sft_a(u)=(u-a)^{-1}=u^{-1}\sum_{k=0}^{\infty}a^k u^{-k}\in R[[u^{-1}]]
\end{equation}
\emph{the formal resolvent} of $a$.
\end{definition}
There is a well known characterization of the minimal polynomial
in terms of the {formal resolvent}.
\begin{prop}\label{prop:min_pole}
Let $a\in R$ and $p\in \Comp[u]$ be a polynomial in one variable. 
Then $p(a)=0$ if and only if $p(u)\sft_a(u)$  is a polynomial in $u$
(a priori, it is only a Laurent formal power series in $u^{-1}$). 
Moreover, both these conditions are equivalent to the vanishing of
the coefficient of $u^{-1}$ in the expansion of $p(u)\sft_a(u)$. 
\end{prop}
\begin{proof}
It is well known (and easy to check) that $(p(u)-p(v))(u-v)^{-1}$ is a 
polynomial in $u$ and $v$. Therefore, 
$p(u)\sft_a(u)=p(u)(u-a)^{-1}=p(a)(u-a)^{-1}+$ ({a polynomial in} $u$). 
From the expansion \eqref{eq:resolve} of $(u-a)^{-1}$ in powers 
of $u^{-1}$ we find that the coefficient of $u^{-k-1}$ in 
$p(u)\sft_a(u)$ is $p(a)a^k$. This proves the proposition.
\end{proof}
\begin{cor}
An element $a\in A$ admits a minimal polynomial if and only if
its formal resolvent $\sft_a(u)$ is a rational function of $u$
with denominator in $\Comp[u]$. 
\end{cor}
We are particularly interested in the case where $R$ is the
algebra of $N\times N$ matrices over another algebra $A$.  
\begin{cor}
Suppose that $a$ is an $N\times N$ matrix over $A$.
Then the minimal polynomial of $a$, viewed as an element
of $R=\sfm_N(A)$, is the monic least common 
denominator in $\Comp[u]$ of the entries of the matrix $\sft_a(u)$.
\end{cor}
\begin{remark}
We may interpret the proposition as saying that the multiset of 
the roots of the minimal polynomial of $a$ coincides with the multiset 
of the poles of the formal resolvent of $a$,
where both the roots and the poles are considered with their multiplicities.
This formulation has important applications to
functional calculus in Banach algebras. 
\end{remark}
The following construction is due to Oshima in \cite{Oshima_min}, 
cf. \cite{Gould_characteristic, Oda_Oshima}. Fix a
Lie algebra $\gg$, an $N$-dimensional $\gg$-module $\pi$ 
with a basis $\{v_1,\ldots,v_N\}$, 
and a $\gg$-equivariant linear map $p:\End\pi\to\ug$. 
Then the elements $p(\sfe_{ij}), 1\leq i,j\leq N$
(where $\sfe_{ij}$ are the ordinary matrix units)
can be arranged into an $N\times N$ matrix $\sff_{\pi}\in\sfm_N(\ug)$.
In this paper $\gg$ is a reductive Lie algebra with a 
fixed nondegenerate invariant symmetric bilinear form and 
$\pi$ is a completely reducible module, so that 
there is a canonical such map $p$. Let $V$ be a $\gg$-module and
denote by $\sff_{\pi,V}$ the image of the matrix $\sff_{\pi}$ 
in the associative algebra $\sfm_N(\End V)$
induced by the natural homomorphism $\ug\to\End V$.
\begin{definition}\label{def:minpol}
With the assumptions above, the \emph{minimal polynomial} $q_{\pi,V}$ 
is the minimal polynomial of the matrix $\sff_{\pi,V}$.
\end{definition}
Gould proved in \cite{Gould_characteristic} 
that if $\gg$ is a semisimple Lie algebra then there exists
``universal'' polynomial with coefficients in $\zg$ annihilating
$\sff_{\pi,V}$. It follows that if a module $V$ has an infinitesimal 
character then it admits a minimal polynomial in the sense
of Definition \ref{def:genminpol}, i.e. with coefficients in $\Comp$.

Our approach to computing the minimal polynomial of a simple
highest weight module $L(\lam)$ is based on the combination
of the criterion of Proposition \ref{prop:criterion}, which
allows one to determine whether the $ad$-invariant 
subspace of $\ug$ spanned by the entries of the matrix
$f(\sff_{\pi})$ annihilates $L(\lam)$ by considering the
Harish-Chandra projection, and of the relation between
the minimal polynomial and the poles of the formal resolvent
of a matrix given in Proposition \ref{prop:min_pole}.  
Let us summarize these remarks in a proposition.
\begin{prop}\label{prop:technique}
Fix a weight $\lam\in\card$.
Denote by $pr:\ug\to\Comp$ the composition of the Harish-Chandra
projection map $pr_{\gg/\car}$ and the evaluation at $\lam$. 
Then the roots of the minimal polynomial of $L(\lam)$ coincide
with the poles of the Laurent formal power series 
$pr \sft_{\sff_{\pi}}(u)\in\Comp((u^{-1}))$, preserving the
multiplicities. Moreover, a polynomial $f(u)$ is divisible by 
the minimal polynomial of $L(\lam)$ 
if and only if the coefficient of $u^{-1}$ in the
formal power series expansion 
$f(u)pr \sft_{\sff_{\pi}}(u)\in\Comp((u^{-1}))$
is $0$. In this case, the expansion reduces to a polynomial in $u$. 
\end{prop}
We now have in our possession a convenient technique
of computing minimal polynomials of simple highest weight modules.
However, an important obstacle remains:  
the Harish-Chandra projection map $pr_{\gg/\car}$ is notoriously difficult
to compute explicitly. (This reflects upon the
complexity of the structure of the primitive spectrum). 
At this juncture, we specialize to 
$\pi=\Comp^N$, the defining module for $\gg$, for 
a symplectic or orthogonal Lie algebra $\gg$ 
and the contragredient ${\Comp^N}^*$ for
a general linear algebra $\gg=\glN$.
In these cases we are able to employ the relative Harish-Chandra projections
arising from the canonical chain of reductive Lie subalgebras
of $\gg$ of decreasing rank (as in the Gelfand--Tsetlin construction
and its analogues for symplectic and orthogonal Lie algebras)
and compute the absolute projection in an inductive fashion using the 
composition formula \eqref{eq:composition}.
\begin{remark}
The arguments just given let one hope that similar computations
can be carried out for any weight multiplicity-free module $\pi$.
Such modules exist for semisimple Lie algebras of all types except
$F_4$ and $E_8$. For the case of a general linear algebra,
we have successfully applied this technique to determination of 
the \emph{quantized elementary divisors}, see Section \ref{sec:primitive}
and \cite{Protsak_divisor}.
\end{remark}
%
%
\section{Minimal polynomials: general linear case}
%
%
Let $\gg=\gln$ be the Lie algebra of $n\times n$ matrices.
The algebra $\gg$ acts on its defining module $\Comp^n$,
an $n$-dimensional complex vector space with a basis $\{v_1,\ldots,v_n\}$.
A basis of $\gln$ is given by the matrix units $\sfe_{ij}, 1\leq i,j\leq n$.
They satisfy the following commutation relation:
\begin{equation}\label{eq:sfe_comm}
[\sfe_{ij},\sfe_{kl}]=\del_{jk}\sfe_{il}-\del_{li}\sfe_{kj}.
\end{equation}
We arrange the generators $\{\sfe_{ij}\}$ into a single 
$n\times n$ matrix $\sfe\in\sfm_N(\ugn)$:
\begin{equation}\label{eq:sfe}
\sfe=\begin{bmatrix}
\sfe_{11} & \sfe_{12} & \ldots & \sfe_{1n} \\
\sfe_{21} & \sfe_{22} & \ldots & \sfe_{2n} \\
\ldots & \ldots & \ldots & \ldots \\
\sfe_{n1} & \sfe_{n2} & \ldots & \sfe_{nn} 
\end{bmatrix}\in\mat_n(\ugn).
\end{equation}
Choose a triangular decomposition $\gg=\nil^{+}\oplus\car\oplus\nil^{-}$, 
where $\nil^{+}$ is spanned by $\{\sfe_{ij}: i<j\}$, $\nil^{-}$ is spanned
by $\{\sfe_{ij}: i>j\}$ and $\car$ is spanned by $\{\sfe_{ii}\}$. 
The elements $\sfh_1=\sfe_{11}, \ldots,\sfh_n=\sfe_{nn}$ form a basis
of $\car$, and we identify the weights $\lam\in\card$ with the $n$-tuples 
$(\lam_1,\ldots, \lam_n)$ of their values on the elements of this basis:
\begin{equation}
\lam_i=\lam(\sfh_i)=\lam(\sff_{ii}).
\end{equation}
We also set 
\begin{equation}
\rho=(n-1,\ldots,0), {\ } 
l=\lam+\rho=(\lam_1+\rho_1,\ldots, l_n+\rho_n), \quad 
\textrm{where } l_i=\lam_i+n-i.
\end{equation}
Denote by $\sft(u)$ the formal resolvent of $\sfe$ (Definition 
\ref{def:resolvent}), viewed as an element of the algebra of 
$n\times n$ matrices over $\ugn$. Then
\begin{equation}\label{eq:res_gl}
\sft(u)=(uI_n-\sfe)^{-1}=u^{-1}+\sum_{k\geq 1}\sfe^k u^{-k-1}
\in\sfm_n(\ugn)[[u^{-1}]].
\end{equation}
\begin{prop}\label{prop:comrel_gl}
The entries of the matrix $\sft(u)$ transform under
the adjoint action of $\gln$ in the same way 
as the entries of the matrix $\sfe$. More explicitly,
\begin{equation}\label{eq:sft_gl}
[\sfe_{ij},\sft_{kl}(u)]=\del_{jk}\sft_{il}(u)-\del_{li}\sft_{kj}(u).
\end{equation}
\end{prop}
\begin{proof}
This follows from the standard commutator identity
\begin{equation}
[x,(X^{-1})_{kl}]=-\sum_{p,q}(X^{-1})_{kp}[x,X_{pq}](X^{-1})_{ql},
\end{equation}
where $X$ is an invertible $n\times n$ matrix with entries in an
associative algebra, $x$ is an element of the same algebra,
and the indices $l,m,p,q$ are between $1$ and $n$.
We take $x=\sfe_{ij}$, $X=uI_n-\sfe$, and use the commutation relations
\eqref{eq:sfe_comm}. A more conceptual proof can be given using the 
adjoint action of the group $GL_n$ on the linear span of the matrix entries 
of the $n\times n$ matrix $uI_n-\sfe$ and of its inverse $\sft(u)$.
\end{proof}
Our goal is to determine the minimal polynomial of an arbitrary simple
highest weight $\gln$-module following the strategy outlined in the 
previous section. We begin by computing some relative Harish-Chandra 
projections with respect to the maximal proper 
standard parabolic subalgebra with the Levi
factor $\gl_{n-1}\oplus\gl_1$.

Let $S$ be the complement of the 
last simple root in the root system $B$ of $\gln$ (in the usual enumeration). 
The parabolic subalgebra $\pp_S=\nil^{+}_{S}\oplus\lev_S$ 
is the $(n-1)$st standard 
maximal parabolic of $\gln$ and consists of the matrices stabilizing 
the codimension one subspace of $\Comp^n$ spanned by the first
$n-1$ standard basis vectors. We have a triangular decomposition
$\gg=\nil^{+}_{S}\oplus\lev_S\oplus\nil^{-}_{S}$, where  
the Levi factor
$\lev_S=\gl_{n-1}\oplus\gl_1$ is spanned by $\sfe_{ij}$
($1\leq i,j\leq n-1$) together with  $\sfe_{nn}$, the upper nilradical
$\nil^{+}_{S}$ is spanned by
$\sfe_{i,n}$ ($1\leq i\leq n-1$) and the lower nilradical $\nil^{-}_{S}$ 
is spanned by  $\sfe_{n,j}$ ($1\leq j\leq n-1$). 

Denote by $\sfe'$ the upper left corner $(n-1)\times(n-1)$ submatrix of 
$\sfe$ formed by the elements $\sfe_{ij}$ 
with $1\leq i,j\leq n-1$ and by $\sft'(u)$ its formal resolvent, 
cf. \eqref{eq:res_gl}. 
Let us write $\lam=(\lam',\lam_n)$,
so that $\lam'_p=\lam_p$ for $1\leq p\leq n-1$.

Suppose that a weight $\lam\in\card$ has been fixed.
Let $pr:\ugn\to\Comp$ be the composition of the absolute 
Harish-Chandra projection $pr_{\gg/\car}$ and the evaluation homomorphism
$ev_\lam: U(\car)\to\Comp, \sfh_i\mapsto\lam_i$. Let the map  
$pr':U(\gl_{n-1})\to\Comp$ have a similar meaning with 
$\gl_{n-1}$ in place of $\gln$. Consider the ``relative'' map
$pr_n:\ugn\to U(\gl_{n-1})$, which is the composition 
of the relative Harish-Chandra projection 
$pr_{\gg/\lev_S}:\ugn\to U(\gl_{n-1})\otimes U(\gl_1)$ and
the partial evaluation homomorphism $\sfh_n\mapsto\lam_n$. 
By Proposition \ref{prop:composition}, these
maps are related by the composition formula $pr=pr'\circ pr_n$

\begin{prop}\label{prop:pr_gl}
The following formulas describe the effect of the 
relative projection map $pr_n$ on the matrix series $\sft(u)$:
\begin{equation}\label{eq:rel_gl}
pr_n\sft_{nn}(u)=\frac{1}{u-\lam_n}, \quad
pr_n\sft_{ij}(u)=(1-\frac{1}{u-\lam_n})\sft'_{ij}(u-1), 1\leq i,j\leq n-1.
\end{equation}
\end{prop}
\begin{proof}
Since the matrices $\sft(u)$ and $uI_n-\sfe$ are mutually inverse, we have: 
\begin{equation} \label{eq:Tinv}
\sum_{k=1}^{n}\sft_{ik}(u)(u\del_{kj}-\sfe_{kj})=\del_{ij}.
\end{equation}
First, let $i=j=n$. Splitting the sum into two parts, corresponding
to whether $1\leq k\leq n-1$ or $k=n$ and applying 
$pr_n$ to both sides, we get:
\begin{equation*}
-\sum_{p=1}^{n-1}pr_n (\sft_{ip}(u)\sfe_{pn})+
pr_n(\sft_{nn}(u)(u-\sfe_{nn}))=1. 
\end{equation*}
Note that since $\sfe_{pn}\in\nil^{+}_S$,
all terms in the first sum in the left hand side vanish, 
and taking into account that 
$\sfe_{nn}\in\lev_S$ and $pr_n\sfe_{nn}=\lam_n$, 
we obtain $pr_n\sft_{nn}(u)(u-\lam_n)=1$, 
proving the first formula.

Now suppose that $1\leq i,j \leq n-1$. Splitting the sum into two
parts as before and applying $pr_n$, we get:
\begin{equation*}
\sum_{p=1}^{n-1}pr_n (\sft_{ip}(u)(u\del_{pj}-\sfe_{pj}))-
pr_n(\sft_{in}(u)\sfe_{nj}))=\del_{ij}. 
\end{equation*}
Using the commutation relations \eqref{eq:sft_gl}, we may replace
$-\sft_{in}(u)\sfe_{nj}$ with 
$-\sfe_{nj}\sft_{in}(u)+[\sfe_{nj},\sft_{in}(u)]=
-\sfe_{nj}\sft_{in}(u)+\del_{ij}\sft_{nn}(u)-\sft_{ij}(u)$.
Since $\sfe_{nj}\in\nil^{-}_{S}$, $pr_n(\sfe_{nj}\sft_{in}(u))=0$
and after routine transformations, we obtain:
\begin{equation*}
\sum_{p=1}^{n-1} pr_n \sft_{ip}(u)((u-1)\del_{pj}-\sfe_{pj})=
\del_{ij}(1-pr_n \sft_{nn}(u)).
\end{equation*}
Substituting the explicit form of $pr_n\sft_{nn}(u)$ just proved, 
this may be restated in the matrix form as follows:
\begin{equation*}
((u-1)I_{n-1}-\sfe')(pr_n \sft(u))'=(1-\frac{1}{u-\lam_n})I_{n-1}.
\end{equation*}
Multiplying both sides by $\sft'(u-1)$ on the left, we obtain 
the desired formula for $pr \sft_{ij}(u), 1\leq i,j\leq n-1$.
\end{proof}

In order to formulate our main result for
the general linear Lie algebra, we need to introduce certain
combinatorial objects.
\begin{definition}
A \emph{falling sequence} is an arithmetic progression with the
step $-1$.
\end{definition}
Next we construct a canonical decomposition of an arbitrary sequence
$l$ into a disjoint union (or shuffle) 
of falling subsequences, called the \emph{shuffle decomposition}. 
(We shall refer to these subsequences
as the \emph{parts} of the shuffle decomposition.) 
This decomposition is best described
by an algorithm. For an empty sequence the shuffle decomposition
is empty. Suppose that the shuffle decomposition 
for a sequence $l'=(l_1,\ldots,l_{n-1})$ has already been
constructed. Let $S$ be its longest part ending in $l_n+1$.
Then the shuffle decomposition of $l=(l_1,\ldots,l_n)$ is obtained by 
appending $l_n$ to the end of $S$, and leaving the other parts intact.
If no such (non-empty) $S$ exists, we create a new part consisting
of a single term $l_n$. It is clear that all parts constructed by
this algorithm are going to be falling sequences, and the decomposition
has a certain minimality property.
\begin{example}
Let $l=(3,3,2,4,1,3,2,2,1)$. Reading the sequence $l$ from
left to right, we easily obtain the shuffle decomposition:
\begin{equation*}
\begin{aligned}
& \{3\} \to 
  \{3\}\sqcup\{3\} \to
  \{3,2\}\sqcup\{3\} \to
  \{3,2\}\sqcup\{3\} \sqcup\{4\}\to 
  \{3,2,1\}\sqcup\{3\} \sqcup\{4\}\to\\ 
& \{3,2,1\}\sqcup\{3\} \sqcup\{4,3\}\to
  \{3,2,1\}\sqcup\{3\} \sqcup\{4,3,2\}\to 
  \{3,2,1\}\sqcup\{3,2\} \sqcup\{4,3,2\}\to\\
&  \{3,2,1\}\sqcup\{3,2\} \sqcup\{4,3,2,1\}.
\end{aligned}
\end{equation*}
Note that on the last step both $\{4,3,2\}$ and $\{3,2\}$ end in $2$, but
we append $1$ to the former subsequence because it is the longer of
the two.
\end{example}
\begin{remark}
The shuffle decomposition of $l$ can also be constructed by reading
its components from right to left. 
\end{remark}
The following theorem completely determines the minimal polynomial of 
a simple highest weight $\gln$-module.
\begin{theorem}\label{thm:minpol_gl}
Let $L(\lam)$ be the simple highest weight $\gln$-module with
highest weight $\lam, l=\lam+\rho$, and the multiset $A$ consist
of the last (i.e. the smallest) terms of the parts of the 
shuffle decomposition of the $n$-element sequence $l$. 
Then $A$ is the multiset of roots of the minimal polynomial
of $L(\lam)$.
\end{theorem}
\begin{proof}
By Proposition \ref{prop:technique}, we need to relate the 
locations and the maximum orders of the
poles of $pr \sft_{ij}(u), 1\leq i,j\leq n$ to the 
shuffle decomposition of the $n$-element sequence $l=\lam+\rho$. 
This is done by induction in $n$ using the formulas \eqref{eq:rel_gl} for
the relative map $pr_n$ and the composition formula $pr=pr'\circ pr_n$.

From the first formula \eqref{eq:rel_gl} we see 
that the minimal polynomial
of $L(\lam)$ must be divisible by $u-\lam_n$. Let $f(u)=(u-\lam_n)g(u)$.
Suppose that $1\leq i,j\leq n-1$, then from the second formula 
\eqref{eq:rel_gl} we get:
\begin{equation}\label{eq:poles_gl}
f(u) pr\sft_{ij}(u)=(u-\lam_n)g(u) pr'(pr_n \sft_{ij}(u))=
(u-1-\lam_n)g(u) pr'\sft'_{ij}(u-1). 
\end{equation}
By the inductive assumption,
the poles of $pr'\sft'_{ij}(u)$ are 
the last terms of the parts of the shuffle decomposition
of $\lam'+\rho'=(\lam_1+n-2, \ldots, \lam_{n-1})$, thus
the poles of $pr'\sft'_{ij}(u-1)$ are 
the last terms of the parts of the shuffle decomposition
of $l'=(\lam_1+n-1, \ldots, \lam_{n-1}+1)$. There are
two cases to consider: either one of these parts, which we call $S$, ends
in $\lam_n+1$, or none of them does. In the former case, we append
$\lam_n$ to the end of $S$, changing the location of the corresponding
pole from $\lam_n+1$ to $\lam_n$. In the latter case, we create
a new part consisting of $\lam_n$. In either case, we see from 
\eqref{eq:poles_gl} and the algorithm for constructing the shuffle
decomposition 
that as we pass from $l'$ to $l$, 
the relation between the shuffle decomposition and the poles 
of the formal resolvent remains unaffected. 
\end{proof}
\begin{remark}
One may ask whether the other elements of the shuffle decomposition of 
$l=\lam+\rho$ admit a similar interpretation. The answer is affirmative and
is related to the notion of \emph{quantized elementary divisors} of the
module $L(\lam)$, see \cite{Protsak_divisor}.
\end{remark}
%
%
%
\section{Minimal polynomials: symplectic/orthogonal case}
%
%
Following \cite{MNO}, we consider split realizations of orthogonal Lie 
algebras and introduce notation that simultaneously 
incorporates the cases of symplectic and orthogonal Lie subalgebras
of the general linear Lie algebra $\glN$. Many of the
formulas below appear most naturally in the context of
the \emph{twisted Yangians}, introduced by Olshanskii in
\cite{Olshanskii_twisted} and further studied in 
\cite{MNO}, but we will not pursue the connection in this paper.

Let $N=2n$ (even case) or $N=2n+1$ (odd case).
We consider a complex vector space $\Comp^N$ with a basis
$\{v_i: -n\leq i\leq n\}$, where $v_0$ is absent in the even case.
This space is endowed with a non-degenerate bilinear form 
$\langle\cdot,\cdot\rangle_{\pm}$, where the subscript $+$ (respectively, $-$) 
indicates that the form is symmetric (respectively, skew-symmetric):
\begin{equation}
\langle v_i, v_j\rangle_{+}=\del_{i,-j} \quad \textrm{(orthogonal)},\quad
\langle v_i, v_j\rangle_{-}=\sgn i\,\del_{i,-j} \quad \textrm{(symplectic)}.
\end{equation}
Let $\theta_{ij}=1$ in the orthogonal case, and $\theta_{ij}=\sgn i\sgn j$
in the symplectic case. The classical Lie algebra $\gg$ is the subalgebra
of $\glN$ consisting of $N\times N$ matrices skew-symmetric with respect
to the form and is spanned by the following elements:
\begin{equation}\label{eq:Fij}
\sff_{ij}=\sfe_{ij}-\theta_{ij}\sfe_{-j,-i},
\end{equation}
where $\sfe_{ij}$ are the usual matrix units. The generators \eqref{eq:Fij}
satisfy the following commutation relations:
\begin{equation}\label{eq:sff_comm}
[\sff_{ij}, \sff_{kl}]=\del_{kj}\sff_{il}-\del_{il}\sff_{kj}
-\theta_{k,-j}\del_{i,-k}\sff_{-j,l}+\theta_{i,-l}\del_{-l,j}\sff_{k,-i}.
\end{equation}
Arrange $\{\sff_{ij}\}$ into an $N\times N$ matrix $\sff\in\sfm_N(\ug)$,
cf \eqref{eq:sfe}. Choose a triangular decomposition 
$\gg=\nil^{+}\oplus\car\oplus\nil^{-}$, 
where $\nil^{+}$ is spanned by $\{\sff_{ij}: i<j\}$, $\nil^{-}$ is spanned
by $\{\sff_{ij}: i>j\}$ and $\car$ is spanned by $\{\sff_{ii}\}$. 
We have $\sff_{ii}=-\sff_{-i,-i}$, the elements 
$\sfh_1=\sff_{-n,-n}, \ldots, \sfh_n=\sff_{-1,-1}$ form a basis of $\car$,
and we will identify the weights $\lam\in\card$ with the $n$-tuples 
$(\lam_1,\ldots, \lam_n)$ of their values on the elements of this basis:
\begin{equation}
\lam_i=\lam(\sfh_i)=\lam(\sff_{i-n-1,i-n-1}).
\end{equation}
Further, we set $\eps=0$ for $\gg$ even orthogonal, $\eps=\frac{1}{2}$ for
$\gg$ odd orthogonal, and $\eps=1$ for $\gg$ symplectic. The half sum
of the positive roots $\rho$ and the $\rho$-shifted highest weight 
$l=\lam+\rho$ are then given by the $n$-tuples
\begin{equation}
\rho=(\eps+n-1,\ldots,\eps), {\ } 
l=\lam+\rho=(\lam_1+\rho_1,\ldots, l_n+\rho_n), \quad 
\textrm{where } l_i=\lam_i+n-i+\eps.
\end{equation}

Denote by $\sft(u)$ the formal resolvent of the $N\times N$ matrix
$\sff$ over $U(\gg)$.
\begin{prop}
The entries of the matrix $\sft(u)$ transform under the adjoint
action of $\gg$ in the same way as the entries of the matrix
$\sff$. More explicitly,
\begin{equation}\label{eq:sft_gg}
[\sff_{ij}, \sft_{kl}(u)]=\del_{kj}\sft_{il}(u)-\del_{il}\sft_{kj}(u)
-\theta_{k,-j}\del_{i,-k}\sft_{-j,l}(u)+\theta_{i,-l}\del_{-l,j}\sft_{k,-i}(u).
\end{equation}
\end{prop}
\begin{proof}
This is proved in the same way as Proposition \ref{prop:comrel_gl}.
\end{proof}
Let $S$ be the complement of the 
first simple root in the root system $B$ of $\gg$ (in standard enumeration). 
The parabolic subalgebra $\pp_S=\nil^{+}_{S}\oplus\lev_S$ 
is the first standard 
maximal parabolic of $\gg$ and consists of the matrices in $\gg\subset\glN$ 
that stabilize the one-dimensional isotropic subspace of $\Comp^N$ spanned 
by the first standard basis vector $v_{-n}$. 
We have a triangular decomposition
$\gg=\nil^{+}_{S}\oplus\lev_S\oplus\nil^{-}_{S}$, where  
the Levi factor
$\lev_S=\gg'\oplus\gl_1$ is spanned by $\sff_{ij}$
($i,j\ne\pm n$) together with  $\sff_{nn}=-\sff_{-n,-n}$, 
the upper nilradical $\nil^{+}_{S}$ is spanned by
$\sff_{-n,j}$ ($j\ne\pm n$) and the lower nilradical $\nil^{-}_{S}$ 
is spanned by  $\sff_{i,n}$ ($i\ne\pm n$). The Lie algebra
$\gg'$ has the same Dynkin type as $\gg$ (i.e. even orthogonal, odd
orthogonal, or symplectic) but its rank is one smaller than the rank of $\gg$.

Denote by $\sff'$ the $(N-2)\times(N-2)$ submatrix of $\sff$ 
formed by the elements $\sff_{ij}$ with $i,j\ne\pm n$ and by 
$\sft'(u)\in\sfm_{n-1}(U(\gg'))[[u^{-1}]]$ its formal resolvent.
Let $\tr'\sft(u)$ be the sum of the $N-2$ diagonal entries of the $N\times N$
matrix $\sft(u)$ in rows with indices different from $\pm n$. 
Let us write $\lam=(\lam_1,\lam')$, so that $\lam'_p=\lam_p$ 
for $2\leq p\leq n$.
As in the general linear case, the maps $pr:\ug\to\Comp, pr':U(\gg')\to\Comp$
denote the compositions of the absolute Harish-Chandra projections with
the evaluation at $\lam$ and $\lam'$. The relative map
$pr_n:\ug\to U(\gg')$ denotes the composition 
of the relative Harish-Chandra projection 
$pr_{\gg/\lev_S}:\ugn\to U(\gg')\otimes U(\gl_1)$ and
the evaluation homomorphism $\sfh_n\mapsto\lam_n$. 
By Proposition \ref{prop:composition}, these maps are related by 
the composition formula $pr=pr'\circ pr_n$. 
\begin{prop}\label{prop:pr_gg}
Let $v=u-\rho_1=u-(n-1+\eps)$ and $l=\lam+\rho$.
The following formulas describe the effect of the relative projection
map $pr_n$ on the matrix series $\sft(u)$:
\begin{equation}
\begin{aligned}
& pr_n\sft_{nn}(u)=\frac{1}{v+l_1}, \\
& pr_n\sft_{ij}(u)=(1-\frac{1}{v+l_1})\sft'_{ij}(u-1), i,j\ne\pm n, \\
& pr_n\sft_{-n,-n}(u)=\frac{1-\tr'\sft(u)}{v-l_1} \quad 
\textrm(orthogonal{\ }case); \\
& pr_n\sft_{-n,-n}(u)=\frac{1-2(v+l_1)^{-1}-\tr'\sft(u)}{v-l_1} \quad 
\textrm(symplectic{\ }case).
\end{aligned}
\end{equation}
Moreover, in the symplectic and even orthogonal case we
have:
\begin{equation}\label{eq:prtr_even}
pr\tr\sft(u)=\frac{v-2\eps}{v-{1/2}}
\big(1-\prod_{i=1}^n\frac{(v-1)^2-l_i^2}{v^2-l_i^2}\big),
\end{equation}
while in the odd orthogonal case 
\begin{equation}\label{eq:prtr_odd}
pr\tr\sft(u)=\frac{v}{v-{1/2}}-\frac{v-1}{v-{1/2}}
\big(1-\prod_{i=1}^n\frac{(v-1)^2-l_i^2}{v^2-l_i^2}\big).
\end{equation}
\end{prop}
\begin{remark}
When both sides of \eqref{eq:prtr_even} and \eqref{eq:prtr_odd}
are expanded in negative powers of $u$, one recovers 
the Perelomov-Popov formulas that describe the action 
of the \emph{Gelfand
invariants} $\tr\sff^p\in Z(\gg)$ in the simple highest weight
module $L(\lam)$, \cite{Zhelobenko}. 
\end{remark}
\begin{proof}
The proof of the first two formulas is similar to the
proof of the corresponding result (Proposition \ref{prop:pr_gl})
in the general linear case. Therefore, some steps in the proof
will only be sketched. 
Since the matrices $\sft(u)$ and 
$uI_N-\sff$ are mutually inverse, we have:
\begin{equation} \label{eq:Tinv_gg}
\sum_{k}\sft_{ik}(u)(u\del_{kj}-\sff_{kj})=\del_{ij}.
\end{equation}
First, let $i=j=n$. Splitting the sum into parts and
proceeding as before, we find that $pr_n(\sft_{nn}(u)(u-\sff_{nn}))=1$,
therefore, $pr_n T_{nn}(u)(u+\lam_n)=1$, and since $u+\lam_1=v+l_1$,
we have established the first
formula. Now suppose that $i,j\ne \pm n$. Splitting the sum
in the left hand side of \eqref{eq:Tinv_gg}
into three parts, according to whether $k\ne\pm n$, $k=n$ or
$k=-n$ and applying $pr_n$ to both sides, we get:
\begin{equation*}
\sump pr_n\sft_{ip}(u)(u\del_{pj}-\sff_{pj})
-pr_n(\sft_{i,-n}(u)\sff_{-nj})-pr_n(\sft_{in}(u)\sff_{nj})=\del_{ij}.
\end{equation*}
Here $\sump$ indicates the sum over the index $p\ne\pm n$.
Note that since $\sff_{-nj}\in\nil^{+}_{S}$, the middle
term in the left hand side is zero.
Using the commutation relations \eqref{eq:sft_gg}, we may
replace $-\sft_{in}(u)\sff_{nj}$ with 
$-\sff_{nj}\sft_{in}(u)+[\sff_{nj},\sft_{in}(u)]=-\sff_{nj}\sft_{in}(u)
+\del_{ij}\sft_{nn}(u)-\sft_{ij}(u)$.
Since $\sff_{nj}\in\nil^{-}_{S}$, $pr_n(\sff_{nj}\sft_{in}(u))=0$
and after routine transformations, we obtain:
\begin{equation*}
\sump pr_n\sft_{ip}(u)((u-1)\del_{pj})=\del_{ij}(1-pr_n\sft_{nn}(u)).
\end{equation*}
Proceeding in the same way as in the proof of Proposition \ref{prop:pr_gl},
we establish the formula for $pr_n\sft_{ij}(u), i,j\ne\pm n$.
\end{proof}
An important distinction of the symplectic/orthogonal case 
that we are considering now from the general linear case treated
earlier is that $\End\pi=\End\Comp^N$ decomposes into the direct sum 
$\Lambda^{2}\Comp^N\oplus S^{2}\Comp^N$ of $\gg$-submodules.
\begin{definition}\label{def:evenodd}
Fix an infinitesimal character of $\gg$-modules.
A polynomial
$p\in\Comp[t]$ is called \emph{even} (respectively, \emph{odd}) 
modulo center if the identity 
\begin{equation*}
p(\sff)_{ij}=\pm \theta_i\theta_j p(\sff_{-j,-i})
\mod \ann M(\lam) 
\end{equation*}
holds for all $i,j$ with the plus (respectively, the minus) sign.
\end{definition}
\begin{prop}
Suppose that $V$ is a module with infinitesimal character. Then
the minimal polynomial of $V$ is either even or odd.
\end{prop}
%
%
%
In analogy with the general linear case, 
we shall prove that the roots of the 
minimal polynomial of a $\gg$-module $L(\lam)$ are obtained from $\lam$
via a combinatorial algorithm. If $l$ is an $n$-element sequence
$(l_1,\ldots,l_n)$, let $l'=(l_2,\ldots,l_n)$ and
$l^{*}=(-l_n,\ldots,-l_1)$. The \emph{mirror image} of the $k$th
term of the $2n$-element sequence $l\cup l^{*}$ is the $(2n+1-k)$th
term of the sequence, and a similar convention applies to subsequences.
The \emph{shuffle decomposition} of $l\cup l^{*}$ is constructed inductively
and represents this $2n$-element sequence as a disjoint union 
of a number of pairs of falling subsequences (referred to
as the \emph{parts} of the decomposition), with the additional property
that the parts in each pair are the mirror images of each other.
For an empty $l$, the decomposition is
empty. Suppose that the shuffle decomposition of $l'\cup l'^{*}$ has already
been constructed. Let $S$ be its longest part which begins with
$l_1-1$, and $S^{*}$ the mirror image of $S$.
Then appending $l_1$ to the beginning of $S$ and $-l_1$ to the 
end of $S^{*}$ and leaving the other parts intact, we get
the shuffle decomposition of $l\cup l^{*}$. If no such (non-empty) $S$
existed, we create two new one-term parts, $\{l_1\}$ 
and $\{-l_1\}$ instead. The shuffle decomposition is called
\emph{odd} if one of the parts is a falling subsequence of $l$ ending
in $\eps$ (or equivalently, if one of the parts is a falling subsequence
of $l^{*}$ starting with $-\eps$). Otherwise, the shuffle decomposition is
\emph{even}.
\begin{theorem}\label{thm:minpol_gg}
Let $\lam$ be a weight for a classical Lie algebra
$\gg$ and $l=\lam+\rho$. Denote by $\tilde{A}$ the set of
the first terms of the sequences comprising 
the shuffle decomposition of $l\cup l^{*}$. Let $A=\tilde{A}$
if the shuffle decomposition is even, and $A=\tilde{A}\setminus\{\eps\}$
if it is odd. Then the roots of the minimal polynomial of
the simple highest weight $\gg$-module $L(\lam)$ are
$\{n-1+\eps-\alpha:\alpha\in A\}$.
\end{theorem}
\begin{proof}[Proof]
The proof proceeds along the same lines as the proof of Theorem
\ref{thm:minpol_gl}. However, instead of dealing with the poles of
$pr \sft_{-n,-n}(u)$ directly, we need to keep track of
whether the minimal polynomial $f$ of $L(\lam)$ is even or odd modulo 
center, Definition \ref{def:evenodd}.
This allows us to restrict
attention to the poles of $pr \sft_{nn}(u)$ and 
$pr \sft_{ij}(u), i,j\ne \pm n$. 
\end{proof}
%
%
%
\section{Behavior of minimal polynomials under transfer}\label{sec:transfer}
%
%
In \cite{Protsak_Transfer}, Section 5, we considered an algebraic 
version of the Howe duality in the context of modules over Lie algebras
forming a reductive dual pair. We first review its definition and then 
investigate the effect this duality has on the minimal polynomials.   
The main result of this section is a noncommutative 
generalization of a well known property of
matrices over a commutative field: suppose that $A$ and $B$ are 
two rectangular matrices of the same size,
so that $AB^t$ and $A^t B$ are square matrices, then the minimal
polynomials $p_{AB^t}(\lambda)$ and $p_{A^t B}(\lambda)$ are either the same 
or differ by a {factor of $\lambda$.} 

\begin{definition}\label{def:alg_duality}
Let $(\gg, \gp)$ is a reductive dual pair of Lie algebras and
$\A$ be a fixed non-zero module over the corresponding Weyl algebra $\weyl$. 
Then a $\gg$-module $V$ and a $\gp$-module $V'$  
are in the algebraic Howe duality with each other if there exists an 
$\Ad(G)$-invariant and $\Ad(G')$-invariant subspace $\A_0$ of $\A$ such
that the quotient $\A/\A_0$ is isomorphic to $V\otimes V'$ as a $\gg$-module
and a $\gp$-module.
\end{definition}
For an irreducible reductive dual pair of general linear Lie algebras
$(\gln, \glk)$, 
the Weyl algebra $\weyl$ may be identified with the algebra of polynomial
coefficient differential operators on $n\times k$ matrices, corresponding
to the realization $W=\Mkn\oplus\Mnk$ for the symplectic vector space.
We denote by $\sfe$ the $n\times n$ matrix \eqref{eq:sfe} 
for $\gln$ and by $\sfep$ the analogous $k\times k$ matrix for $\glk$.
Let $\sfx$ be the $k\times n$ matrix
with entries $x_{ai}$, the coordinate functions on $\Mkn$ 
and $\sfd$ be the $k\times n$ matrix with entries $\dd_{bj}$,
the corresponding partial derivatives. 
The \emph{unnormalized embeddings} $L$ and $R$ of $\glk$ and $\gln$ into 
$\weyl$ are given by the following explicit formulas:
\begin{equation}\label{eq:gen_gl}
L(\sfe')=\sfx\sfd^t, \quad R(\sfe)=\sfx^t\sfd.
\end{equation}
The action of $\glk$ on the polynomial functions on $\Mkn$ 
via $L$ arises from the 
action of the group $GL_k$ on the $k\times n$ matrices by the left 
matrix multiplication, and similarly for $\gln$, $R$, and the right
matrix multiplication. In the language of the classical invariant
theory, the differential operators $R(\sfe_{ij})$ are the 
\emph{polarization operators}
with respect to the $GL_k$-action and likewise for the
operators $L(\sfep_{ab})$ and the $GL_n$-action.
Any module over the Weyl algebra acquires a structure of a $\gln$-module
and $\glk$-module by pulling back the $\weyl$-action via the
maps $R$ and $L$.
\begin{theorem}
Let $(\gg,\gp)=(\gln,\glk)$ and the modules $V$ and $V'$ over 
$\gln$ and $\glk$ are in an unnormalized algebraic Howe duality with each
other. Then the following divisibility
properties hold for their minimal polynomials $q=q_V$ and $q'=q_{V'}$:
\begin{equation}
q(u) {\ \vert\ } uq'(u+k-n), \quad q'(u) {\ \vert\ } uq(u-k+n). 
\end{equation}
\end{theorem}
The statement of the theorem is symmetric with respect to exchanging
$n$ and $k$, and its conclusion given does not depend on the particular
module over the Weyl algebra used to define algebraic Howe duality.
\begin{proof}
The first part of the proof takes place in the Weyl algebra.
Consider the matrices over $\weyl$ obtained by applying
$R$ and $L$ entriwise to the matrices $\sfe$ and $\sfep$.
By an abuse of notation, we will also denote them $\sfe$ and $\sfep$.
Similar conventions apply to the matrix resolvents $\sft(u)$ and
$\sftp(u)$ of $\sfe$ and $\sfep$. 
\begin{prop}
For any $r\geq 0$, the following identity holds: 
\begin{equation}\label{eq:conv_powers}
\sum_l(\sfe)^r_{il}x_{al}=\sum_b(\sfep+(n-k)I_k)^r_{ab}x_{bi}.
\end{equation}
\end{prop} 
\begin{proof}
We apply induction in $r$. For $r=0$ both sides are equal to $x_{ai}$. For
$r=1$, we compute
\begin{equation*}
\begin{aligned}
\sum_l\sfe_{il} x_{al} & = \sum_{b,l} x_{bi} \dd_{bl} x_{al}=
\sum_{b,l} ( x_{al} \dd_{bl} x_{bi}  + \del_{ab} x_{bi} - \del_{il} x_{al}) 
= \\ & = (n-k)x_{ai} + \sum_b \sfep_{ab} x_{bi} 
= \sum_b  (\sfep+(n-k)I_k)_{ab} x_{bi}.
\end{aligned}
\end{equation*}
If the identity is true for $r\geq 1$ then 
\begin{equation*}
\begin{aligned}
& \sum_l(\sfe)^{r+1}_{il} x_{al} = \sum_{l,m}\sfe_{im}(\sfe)^r_{ml}x_{al} =\\
=& \sum_{c,m}\sfe_{im}(\sfep+(n-k)I_k)^r_{ac}x_{cm} = 
\sum_{c,m}(\sfep+(n-k)I_k)^r_{ac}\sfe_{im} x_{cm} =\\  
=& \sum_{c}(\sfep+(n-k)I_k)^r_{ac}(\sfep+(n-k)I_k)_{cb} x_{bi}
= (\sfep+(n-k)I_k)^{r+1}_{ab} x_{bi}.
\end{aligned}
\end{equation*}
Since each entry of the matrix $\sfe$ commutes with each entry of the 
matrix $\sfep$, the exchange of the $\sfe$-factor and $\sfep$-factor
in the second line is justified.  
\end{proof}
Multiplying the $r$th identity \eqref{eq:conv_powers} by $u^{-1-r}$ 
and summing over $r\geq 0$, we obtain the identity 
\begin{equation}\label{eq:conv1}
\sum_l\sft_{il}(u)x_{al}=\sum_b\sftp_{ab}(u+k-n)x_{bi}, \quad
\sft(u)\sfx^t=(\sftp(u+k-n)\sfx)^{t}.
\end{equation}
Now multiplying both sides by $\dd_{aj}$ and
summing over $a$, we see that 
\begin{equation*}
\sum_l \sft_{il}(u)\sfe_{lj} = \sum_{a,l} \sft_{il}(u)x_{al}\dd_{aj} = 
\sum_{a,b}\sftp_{ab}(u+k-n)x_{bi}\dd_{aj}.
\end{equation*}
Adding $\sum_l\sft_{il}(u)(uI_n-\sfe)_{lj}=(I_n)_{ij}$ to both sides, 
we finally arrive at the identity
\begin{equation*}
u\sft_{ij}(u) = (I_n)_{ij}+\sum_{a,b}\sftp_{ab}(u+k-n)x_{bi}\dd_{aj},\quad
u\sft(u) = I_n + (\sftp(u+k-n)\sfx)^t\sfd.
\end{equation*} 
Now let us consider the quotient $\A/\A_0$ of a module over the Weyl
algebra that realizes the algebraic Howe duality between $V$ and $V'$. 
Since $q'$ is the minimal polynomial of the $\glk$-module $V'$, 
each matrix entry of $q'(u+k-n)\sftp(u+k-n)$ acts 
on this quotient by a polynomial in $u$. 
It follows that each matrix entry of $uq'(u+n-k)\sft(u)$ acts 
on the $\gln$-module $V$ by a polynomial in $u$, establishing that
$uq'(u+n-k)$ is divisible by the minimal polynomial $q(u)$ of the 
$\gln$-module $V$.
\end{proof} 
%
%
\section{Applications to primitive ideals}\label{sec:primitive}
%
%
%
\begin{theorem}
Suppose that a simple highest weight module $L(\lam)$ over $\gln$ has
regular infinitesimal character. Then the minimal polynomial of
$L(\lam)$ is uniquely determined by the tau invariant of $L(\lam)$.
Moreover, for simple highest weight modules in the infinitesimal
character of the trivial representation, the minimal polynomial and
the $\tau$-invariant of $L(\lam)$ completely determine each
other. Inclusions of primitive ideals correspond to divisibility of
their minimal polynomials and inclusions of their tau invariants.
\end{theorem}
\begin{theorem}
Let $\pi$ be the adjoint representation of $\sln$ extended to $\gln$
trivially on the center. Then the multiplicity of $\pi$ in the
primitive quotient $\ugn/\ann L(\lam)$ is one less than the degree of
the minimal polynomial of $L(\lam)$.
\end{theorem}
%
%
%
\section{Concluding remarks}\label{sec:remarks}
%
%
It should be pointed out that in analogy with the classification of
the conjugacy classes of matrices (or of the adjoint orbits in $\gg$),
the minimal polynomial is only the first member in a larger family of
invariants of modules and ideals. Therefore, it is insufficient for
classification of the primitive ideals of $\ug$. In the case of the
general Lie algebra $\gln$, we introduced in \cite{Protsak_divisor} a
family of invariants $q_k, 1\leq k\leq n$ generalizing the minimal
polynomial. Their definition is inspired by Oshima's construction of
explicit generators for the induced ideals of $\ugn$,
\cite{Oshima_quantization}. These invariants may be viewed as
quantizations of the Fitting invariants of $n\times n$ matrices and
analogously to the theory of elementary divisors of matrices, they
separate completely prime primitive ideals of $\ugn$. We hope that
similar generalizations are possible also for the symplectic and
orthogonal Lie algebras.
%
%
%
\bibliographystyle{plain}
\bibliography{refer} 
%

\end{document}